\documentclass{amsart}
\usepackage{amsmath}
\usepackage{amssymb}
\usepackage{graphicx}
\input xy
\xyoption{all}
\usepackage{color}
\usepackage{enumerate}
\usepackage{amsthm,enumitem}
\newtheorem{theorem}{Theorem}

\newtheorem{conjecture}{Conjecture}
\newtheorem{lemma}{Lemma}
\newtheorem{corollary}{Corollary}
\newtheorem{proposition}{Proposition}
\theoremstyle{definition}
\newtheorem{definition}{Definition}
\newtheorem{remark}{Remark}

\begin{document}
	
	\title[Subgroups in weakly locally finite  division rings]{Multiplicative Subgroups in\\ weakly locally finite division rings}\thanks{This work is funded by Vietnam National Foundation for Science and Technology Development (NAFOSTED) under Grant No. 101.04-2019.323.}
	
	\author[Bui Xuan Hai]{Bui Xuan Hai $^{1,2}$}
	
	\author[Huynh Viet Khanh]{Huynh Viet Khanh$^{1, 2}$}
	\address[1]{Faculty of Mathematics and Computer Science, University of Science;}
	\address[2]{Vietnam National University, Ho Chi Minh City, Vietnam.}
	\email{ bxhai@hcmus.edu.vn; huynhvietkhanh@gmail.com}
	
	\maketitle
\begin{abstract} The description of the subgroup structure of a non-commutative division ring is the subject of the intensive study in the theory of division rings in particular, and of the theory of skew linear groups in general. This study is still so far to be complete. In this paper, we study this problem for weakly locally finite division rings. Such division rings constitute a large class which strictly contains the class of locally finite division rings. 

\end{abstract}

\section{Introduction}
In the study of the structure of non-commutative rings, division rings and their matrix rings can be considered as basic subjects. Among non-commutative rings, division rings constitute the simplest class. However, much less is known about their structure although they are ``perfect" subjects in which we can not only add, subtract, multiply, but also divide by non-zero elements. In the commutative case, the theory of fields  and their matrix rings is extensively developed branch of algebra with  very  nice results achieved even before the end of the last century. If every undergraduate student can easily give various examples of fields, then it is not the case with non-commutative division rings. The first  non-commutative division ring was discovered in 1843 by sir William Rowan Hamilton, an Irish mathematician when he studied Mechanics. This division ring is now called the division ring of real quaternions or also the Hamilton division ring. Ignore the rich application of Hamilton quaternions in Mechanics and Physics, the Hamilton discovery serves a very nice (and first) example of a non-commutative division ring. At present, infinitely many non-commutative division rings exist. Indeed, one can construct new non-commutative division rings by different manner from previous ones. Notice only, to possess these constructions, there are now several good books to read (e.g. see \cite{cohn}, \cite{lam}  etc.) In 1905, J. Wedderburn, a British mathematician, proved a famous theorem which is now referred to as the Wedderburn Little Theorem, stating that \textit{``Every finite division ring is commutative"}. This nice theorem serves as the inspiration for many further investigations to generalize it, and later, it was established even one direction of study in Ring Theory called \textit{``Commutativity theorems"} (e.g. see \cite[Chapter 3]{book_her}). Let $D$ be a division ring with center $F$ and the multiplicative group $D^*$. We may think about  the Wedderburn Little Theorem as the following: \textit{``If the multiplicative group $D^*$ of a division ring $D$ is finite, then $D=F$; that is, $D$ is a field".} With this viewpoint, we may think that the group properties of $D$ may influence to its algebraic properties. In the reality, many authors make the great attention to this side of the matter. For instance, in 1945, N. Jacobson \cite[Theorem 9]{1945_jacob} proved that if every element of $D^*$ is algebraic over a finite subfield of $F$, then $D$ is commutative. In 1951, I. Kaplansky \cite{1951_kap} showed that if the factor group $D^*/F^*$ is periodic, then $D=F$. In 1950, L. K. Hua \cite{1950_hua} proved that if $D^*$ is solvable, then $D$ is a field. Now, assume that $G$ is an arbitrary subnormal subgroup of  $D^*$. In 1964, Stuth \cite{1964_stuth} extended Hua's result by proving that if  $G$ is solvable, then $G$ is contained in $F$. Later, many authors study this phenomenon, and there is now a lot of published papers devoted to this direction of research (e.g. see \cite{2003_akbari}, \cite{chiba}, \cite{deo-bien-hai-20}, \cite{deo19}, \cite{2012_JAlgebra_dbh}, \cite{2013_hai-thin}, \cite{2009_hai-thin}, \cite{2010_hai-ha}, \cite{her_1978}, \cite{2019_khanh-hai}, and references therein).   

The aim of this paper is to study subgroup structure of weakly locally finite division rings. Such rings were firstly introduced in \cite{2012_JAlgebra_dbh} and some their basic properties were studied in \cite{deo19}, \cite{2012_JAlgebra_dbh}, \cite{2013_hai-ngoc}. By results from \cite{deo19}, we know that the class of weakly locally finite division rings strictly contains the class of locally finite division rings. For further study of the subject of this paper, in Section~2, we give the relationship among some classes of division rings. Among them, there is a class lying between the two classes we have mentioned above, namely, the class of   Stewart's division rings. Section 3 spends to the proof of a theorem which can be considered as the Tits Alternative for weakly locally finite division rings. This theorem will be used frequently in the next sections of the paper. Section 4 devotes to the study of the relationship between the class of Stewart's division rings and the class of weakly locally finite division rings. The main result in this section states that the problem on the existence of non-cyclic free subgroups in weakly locally finite division rings can be reduced to that in Stewart's division rings.  Finally, in Section 5, we study maximal subgroups in an almost subnormal subgroup of a weakly locally finite division ring. In this last section, the problem on the existence of non-cyclic free subgroups remains also attentive point. The new results we get here generalize the  results previously obtained in \cite{2019_hai-khanh}, \cite{2016_hai-tu},  \cite{2001_mah}.

Throughout this paper, all  symbols and notations we use are standard according to the literature in the area of study. In particular, if $S$ is a non-empty subset of a division ring $D$ and $E$ is a division subring of $D$, then $E[S]$ and $E(S)$ denote the subring and the division subring of $D$ generated by the set $E\cup S$ respectively.  If $R$ is an associative ring with identity, then the symbol $R^*$ stands for the group of units of $R$; that is, the multiplicative group consisting of all invertible elements of $R$. If $X$ is either a ring or a group, then $Z(X)$ denotes the center of $X$.
\section{Some classification of division rings}

The description of subgroup structure of skew linear groups in general depends essentially on base division rings. By definition, a division ring $D$ is \textit{centrally finite} if $D$ is a finite dimentional vector space over its center $F$; that is, $[D:F]<\infty$.  Centrally finite division rings constitute the simplest class of division rings because of the existence of the regular representation. Indeed, if  $[D:F]=n$, then it is well-known that via regular representation, the multiplicative group $D^*$ of $D$ can be viewed as a subgroup of the general linear group $\mathrm{GL}_n(F)$ over the field $F$. Consequently, every skew linear group of degree $m$ can be viewed as a subgroup of $\mathrm{GL}_{nm}(F)$; that is, it can be viewed as a linear group over a field. So, in this case, to study skew linear groups of degree $m$, one can use the methods and results from the theory of linear groups over fields. Of course, the case of division rings which are infinite dimensional over their centers is more complicated, and much less we know about their structures. Hence, it is more reasonable to classify division rings into classes from more simple classes into more complicated ones, and then study their structures in particular, and structures of skew linear groups over them in general. Let us recall some classical definitions. Let $D$ be a division ring with center $F$. If for any finite subset $S$ of $ D$, the division subring $F(S)$ of $D$ generated by the set $S\cup F$ is a finite dimensional vector space over $F$, then $D$ is called \textit{locally finite}. Clearly, a centrally finite division ring is locally finite. There exist infinitely many locally finite division rings that are not centrally finite (e.g. see \cite[$\S$ 13]{lam}). An element $a\in D$ is \textit{algebraic} over $F$ if $f(a)=0$ for some non-zero polynomial $f(t)\in F[t]$. A subset $S$ of $D$ is \textit{algebraic} over $F$ if $a$ is algebraic over $F$ for every $a\in S$. In particular, $D$ is \textit{algebraic} over $F$ (shortly, $D$ is \textit{algebraic}) if every element of $D$ is algebraic over $F$. Clearly, every locally finite division ring is algebraic. In 1941, Kurosh \cite[Problem K]{kurosh} asked whether every algebraic division ring is locally finite? At  present, this problem remains still unsolved in general: there are no examples of algebraic division rings that are not locally finite on one side, and on the other side, it is answered in the affirmative for all known up to now special cases of a division ring $D$ with the center $F$.  In particular, it is the case: for $F$ finite by Jacobson \cite{1945_jacob}; for $F$ having only finite algebraic field extensions (in particular for $F$ algebraically closed). The last case follows from the Levitzki-Shirshov theorem which states that {\em any algebraic algebra of bounded degree is locally finite} (e.g. see  \cite{khar}). Recently, in \cite{deo19}, \cite{2012_JAlgebra_dbh},  \cite{2013_hai-ngoc}, the authors have introduced the notion of weakly locally finite division rings and studied their basic properties. A division ring $D$ with center $F$ is  \textit{weakly locally finite} if for any finite subset $S$ of $D$, the division subring of $D$ generated by $S$ is centrally finite (observe that if $\mathbb{P}$ is the prime subfield of $D$, then $\mathbb{P}(S)$ is exactly the division subring of $D$ generated by $S$). It is easy to prove that any locally finite division ring satisfies the last property, so every locally finite division ring is weakly locally finite (the detail of the proof can be seen in \cite{2013_hai-ngoc}). In \cite{deo19}, the authors have constructed for any positive integer $n$ a division ring $D$ with the Gelfand-Kirillov dimension ${\rm GKdim} (D)=n$ which is not algebraic. It is well-known that the Gelfand-Kirillov dimension of a locally finite division ring is zero \cite{lenagan}. Hence, the examples from \cite{deo19} show that there exist infinitely many weakly locally finite division rings that are not locally finite. In 1986, in his Ph.D. thesis \cite{stewart}, A. I. Stewart introduced the class of division rings which is wider than the class of locally finite division rings. In the following, Stewart's original definition is given. 

\begin{definition}\label{definition_Stewar's}
	A division ring $D$ is called \textit{locally finite left dimensional over a subfield $E$} if for every finite subset $S$ of $D$, the division subring $E(S)$ of $D$ generated by the set $E\cup S$ is finite dimensional left vector space over $E$. The concept of \textit{locally finite right dimensional over a subfield $E$} is defined similarly.
\end{definition}
Keeping notations and symbols as in Definition \ref{definition_Stewar's}, it follows  by \cite[Lemma~ 2.1]{stewart_87} that $E(S)$ is a centrally finite division ring.  Moreover, both the left and right dimensions of $E(S)$ over $E$ are finite and  coincide; that is,   $$[E(S):E]_R=[E(S):E]_L<\infty.$$
Hence, a division ring $D$ is locally finite left dimensional over a subfield $E$ if and only if it is locally finite right dimensional over $E$. In view of this observation, if either of these conditions holds, then we say simply that $D$ is \textit{locally finite over a subfield $E$}. For short, we call such a division ring  a \textit{Stewart's division ring}. For the convenience, we summarize this fact in the following lemma.
\begin{lemma}[{\cite[Lemma 2.1]{stewart_87}}] \label{lem:2.2} Let $D$ be a division ring which is locally finite over a subfield $E$. Then, the following assertions hold:
	\begin{enumerate}
		\item For every finite subset $S$ of $D$, the division subring $E(S)$ of $D$ generated by the set $E\cup S$ is centrally finite.
		\item The left and right dimensions of $E(S)$ over $E$ are both finite and coincide; that is 
		$$[E(S):E]_R=[E(S):E]_L<\infty.$$ 
	\end{enumerate}
\end{lemma}
\begin{remark}\label{rem:1} Let  $D$ be a division ring  which is locally finite over some subfield $E$. If $E$ coincide with the center $F$ of $D$, then it is evident that $D$ is locally finite. In general, $E$ may not equal to $F$. Sometimes, for the convenience, we use the expression \textit{``$D$ is a $E$-Stewart's division ring"}  to indicate the fact that $D$ is locally finite over a subfield $E$.
\end{remark}

From the definitions we see that every  locally finite division ring is a Stewart's division ring. The fact that the class of Stewart's division rings strictly contains the class of locally finite division rings can be seen from the following two results of Stewart.
\begin{theorem}[{\cite[Theorem A]{stewart_87}}]\label{th:2.2} Let $D$ be a locally finite division ring with center $F$, $n$ a positive integer, and $G$ a supersolvable subgroup of $\mathrm{GL}_n(D)$. Then, the following statements hold:
	\begin{enumerate}
		\item $G$ is locally nilpotent-by-periodic abelian, but not necessarily locally nilpotent-by-finite.
		\item  $G$ is nilpotent-by-locally finite.
		\item If either $u(G)=\langle 1\rangle$ or $F$ is perfect of characteristic $2$, then $G$ is hypercyclic.
	\end{enumerate}
\end{theorem}
\begin{theorem}[{\cite[Theorem 1]{stewart_87_1}}]\label{th:2.3} Let $p$ be a prime number. Then, there exists a Stewart's division ring $D$ of characteristic $p$, and a locally nilpotent primary subgroup $G$ of $\mathrm{GL}_2(D)$ such that $G$ does not satisfy the normalizer condition and hence is not hypercentral.
\end{theorem}
Since the subgroup $G$ of $\mathrm{GL}_2(D)$ in Theorem \ref{th:2.3}  is not hypercentral, it is not hypercyclic either. It follows from Theorem \ref{th:2.2} that $D$ is not locally finite. 

From the above discussion, we see that both classes of weakly locally finite division rings and Stewart's division rings strictly contains the class of locally finite division rings. It is  natural to ask what is the relationship between these two classes? By the following theorem, we show that the class of Stewart's division rings is contained in the class of weakly locally finite division rings. 

\begin{theorem}\label{theorem_2.2}
	Every Stewart's division ring is weakly locally finite.
\end{theorem}
\begin{proof}
	Assume that $D$ is a Stewart's division ring with the prime subfield $\mathbb{P}$. Then, there exists a subfield $E$ of $D$ such that $D$ is locally finite  over $E$. To prove that $D$ is weakly locally finite, we must to show that $\mathbb{P}(S)$ is centrally finite for an arbitrary finite subset $S$ of $D$. It follows from  Lemma \ref{lem:2.2} that $E(S)$ is a centrally finite division ring which contains $\mathbb{P}(S)$. By the fact that any division subring of a centrally finite division ring is also centrally finite \cite[Theorem 3]{2013_hai-ngoc}, we conclude that $\mathbb{P}(S)$ is centrally finite, as desired. 
\end{proof}
The above theorem raises a significant question: whether its converse statement holds? In an attempt to find the answer to this question, we are feeling that an arbitrary weakly locally finite division ring is very far from being locally finite over some subfield. We shall discuss further in Section 4 some necessary conditions for a weakly locally finite division ring to be a Stewart's division ring. Meantime, this is a good moment to pose the following conjecture.

\begin{conjecture}\label{conj:2.1}
	The class of weakly locally finite division rings strictly contains the class of Stewart's division rings. 
\end{conjecture}

Before going on, we give the following chart to summarize the relationship among some classes of division rings. 

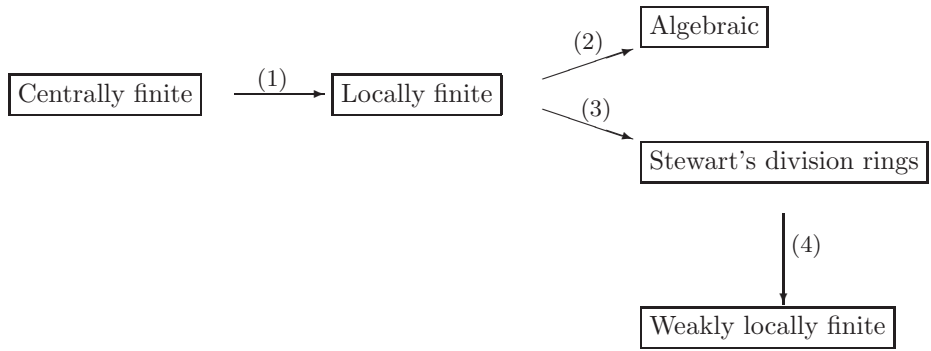
\begin{center}
	\unitlength=1mm
	\special{em:linewidth 0.4pt}
	\linethickness{0.4pt}
	\begin{picture}(120.00,20.00)(0,0)
	\put(0,5){\fbox{Centrally finite}}
	\put(30,6){\vector(1,0){12}}
	\put(33,7){\small{(1)}}
	\put(43,5){\fbox{Locally finite}}
	\put(71,8){\vector(3,1){12}}
	\put(71,4){\vector(3,-1){12}}
	\put(75,12){\small{(2)}}
	\put(76,3){\small{(3)}}
	\put(84,14){\fbox{Algebraic}}
	\put(84,-4){\fbox{Stewart's division rings}}
	\put(103,-10){\vector(0,-1){12}}
	\put(104,-15){\small{(4)}}
	\put(84,-26){\fbox{Weakly locally finite}}
	
	\end{picture}
\end{center}
\vspace*{3cm}

In this chart, the arrows mean the inclusions. By the discussion above, we see that the arrows (1) and (3) are not reversible, while the reversibility of the arrow (4)  is Conjecture \ref{conj:2.1}, and the question whether arrow (2) is reversible is exactly the Kurosh Problem for division rings.  
\section{An analogy of the Tits Alternative}

In 1972, J. Tits \cite{tits} proved a famous result on linear groups over fields, which is now often referred to as the Tits Alternative. For the convenience, we restate Tits' result in the following theorem.

\bigskip
\noindent
\textbf{Theorem A (the Tits Alternative)}. \textit{Let $F$ be a field, $n$ a positive integer and $\mathrm{GL}_n(F)$ the general linear group of degree $n$ over $F$.}
\textit{	\begin{enumerate}
		\item If ${\mathrm char}(F)=0$, then every subgroup of $\mathrm{GL}_n(F)$ either contains a non- cyclic free subgroup or is solvable-by-finite.
		\item If ${\mathrm char}(F)=p>0$, then every finitely generated subgroup of $\mathrm{GL}_n(F)$ either contains a non-cyclic free subgroup or is solvable-by-finite.
	\end{enumerate}}
	\bigskip
	In the theorem above, if we replace $F$ by a non-commutative division ring $D$ then in general, the conclusions fail to be true (e.g. see \cite{lichtman} for the case $n=1$; that is, for $D^*=\mathrm{GL}_1(D)$). The Tits Alternative served the starting point for many works focusing on finding the answer for the question whether the multiplicative group $D^*$ of a non-commutative division ring $D$ contains non-cyclic free subgroups. Although many efforts have been made, the question  remains still without the definitive answer in general case. Whereas, there is a lot of works giving an affirmative answer for this question in various particular cases. For instance, it was shown that the question has a positive answer in the following cases: $D$ is centrally finite \cite{goncalves_86}, $D$ is weakly locally finite \cite{2013_hai-ngoc}, \cite{ngoc_bien_hai_17}, and $D$ has uncountable center \cite{chiba}. These works also demonstrate that the most of these results  hold for any non-central  subnormal subgroup, instead of $D^*$.
	
	The aim of this section is to prove a theorem which is similar to Theorem A when a field $F$ is replaced by a weakly locally finite division ring $D$. The result obtained in Theorem \ref{theorem_2.8} will be used in the next sections to study the problem on the existence of non-cyclic free subgroups in weakly locally finite division rings, and, more generally, in  skew linear groups over them. The proof we develop  involves the technique of the Zariski topology and $CZ$-groups introduced in \cite[Chapter 5]{weh_1973}. 
	
	Let $F$ be a field. Then, by viewing as a vector space (of dimension $n^2$) over $F$, the matrix ring ${\mathrm M}_n(F)$ carries the Zariski topology. If $G$ is a subgroup of $\mathrm{GL}_n(F)$, then the induced topology on $G$ make it into a $CZ$-group. For any field extension $F\subseteq K$, the tensor product ${\mathrm M}_n(F)\otimes_F K\cong {\mathrm M}_n(K) $ is a finite-dimensional $K$-space. Therefore, it carries the Zariski topology as $K$-space and this induces a topology on ${\mathrm M}_n(F)\otimes_F 1$, and thus on ${\mathrm M}_n(F)$. It can be shown that these two topologies on ${\mathrm M}_n(F)$ are homeomorphic (\cite[p.73]{weh_1973}).
	
	Let $D$ be a weakly locally finite division ring with center $F$. By definition, for any finitely generated subgroup $Y$ of $D^*$, the division subring $D_Y=F(Y)$ of $D$ is centrally finite. Therefore, if we denote by $F_Y$  the center of $D_Y$, then the dimension $m_Y=[D_Y:F_Y]$ is finite, and the matrix ring  ${\mathrm M}_{m_Y}(F_Y)$ is a topological space with the Zariski topology. 
	
	If we view $Y$ as a subgroup of $\mathrm{GL}_{m_Y}(F_Y)$, then the Zariski topology on ${\mathrm M}_{m_Y}(F_Y)$ makes $Y$ into a $CZ$-group. It is clear that every subgroup $X$ of $Y$ is a $CZ$-group too. For such a subgroup $X$, denote by $X^{0}_Y$ the connected component of $X$ containing the identity $1$ with respect to the Zariski topology defined on ${\mathrm M}_{m_Y}(F_Y)$. For a subgroup $G$  of $D^*$, set $G^+:=\cup_Y(G\cap Y)^{0}_Y$, where $Y$ runs over the set of all finitely generated subgroups of $D^*$. 
	We shall prove that $G^+$ is a normal subgroup of $G$ and $G/G^+$ is locally finite. For this purpose, we need some lemmas.
	
	\begin{lemma}\label{lem:3.1} 
		Let $D$ be a division ring with center $F$. Assume that $S_1\subseteq S_2$ are two finite subsets of $D$. If $F_1$ and $F_2$ are the centers of $D_1=F(S_1)$ and $D_2=F(S_2)$ respectively, then there exists a subfield of $D$ containing both $F_1$ and $F_2$. 
	\end{lemma}
	
	\begin{proof} 
		Clearly, the inclusion $S_1\subseteq S_2$ implies $D_1\subseteq D_2$. If we set $$D_3=C_{D_2}(D_1)=\{x\in D_2~\vert~ xy=yx, \hbox{ for all } y\in D_1\},$$ and $F_3$ to be the center of $D_3$, then $F_3$ is a subfield of $D$ containing both $F_1$ and $F_2$. 
	\end{proof}
	
	\begin{lemma}\label{lemma_2.5}
		Let $D$ be a weakly locally finite division ring, and $G$ a subgroup of $D^*$. Suppose that $Y_1$ and $Y_2$ are finitely generated subgroups of $D^*$. If $x_1\in (G\cap Y_1)^0_{Y_1}$ and $x_2\in (G\cap Y_2)^0_{Y_2}$, then there exists a finitely generated subgroup $Y$ of $D^*$ such that $(G\cap Y)^0_{Y}$ contains both $x_1$ and $x_2$.
	\end{lemma}
	\begin{proof}
		
		Let us keep the symbols $D_Y, F_Y$ and $m_Y$ for any finitely generated subgroup $Y$ of $D^*$ as above. If we set $Y=\left\langle Y_1,Y_2\right\rangle \leq D^*$, then $Y$ is a $CZ$-group with the Zariski topology on ${\mathrm M}_{m_Y}(F_Y)$. Let $m=\max\{m_{Y_1},m_{Y}\}$. By Lemma \ref{lem:3.1}, there exists a subfield $K$ of $D_Y$ containing both $F_{Y_1}$ and $F_{Y}$. Hence, $Y_1\le Y$ may be viewed as subgroups of $\mathrm{GL}_{m}(K)$. As we have mentioned above, the topology on $Y$ (resp. $Y_1$), which is induced from the Zariski topology on ${\mathrm M}_m(K)$, coincides with the topology induced from that on  ${\mathrm M}_{ m_{Y}}(F_{Y})$ (resp. ${\mathrm M}_{m_{Y_1}}(F_{Y_1})$). Therefore, the condition $x_1\in (G\cap Y_1)_{Y_1}^0$ implies  $x_1\in (G\cap Y)_Y^0$. By a similar argument, we also have $x_2\in (G\cap Y)_Y^0$. 
	\end{proof}
	
	\begin{theorem}\label{th:3.3} $G^+$ is a normal subgroup of $G$ and $G/G^+$ is locally finite.
	\end{theorem}  
	\begin{proof} 
		It is clear that $G^+$ contains $1$. For arbitrary elements $x_1$ and $x_2$ of $G^+$, there exist finitely generated subgroups $Y_1$ and $Y_2$ of $D^*$ such that $x_1\in (G\cap Y_1)_{Y_1}^0$ and $x_2\in (G\cap Y_2)_{Y_2}^0$.  In view of Lemma \ref{lemma_2.5}, we can find a finitely generated subgroup $Y_0$ of $D^*$ such that $x_1, x_2 \in (G\cap Y_0)_{Y_0}^0$. Since $(G\cap Y_0)_{Y_0}^0$ is in fact a group, we have $x_1x_2\in (G\cap Y_0)_{Y_0}^0 \le G^+$, from which it follows that $G^+$ is closed under multiplication. A simple argument establishes that $G^+$ is also closed under taking of inverse elements. Hence, $G^+$ forms a subgroup of $G$. Now, assume that $g\in G$ and $x\in G^+$. By definition of $G^+$, there exists a finitely generated subgroup $Y$ of $D^*$ such that $x\in(G\cap Y)_Y^0$. Let $Y_g$ be the subgroup of $D^*$  generated by $Y\cup\{g\}$. According to Lemma \ref{lemma_2.5}, we deduce that $x\in (G\cap Y_g)_{Y_g}^0$, and so $g^{-1}xg\in (G\cap Y_g)_{Y_g}^0$. Because $(G\cap Y_g)_{Y_g}^0\le G^+$, we have $g^{-1}xg\in G^+$, proving the normality of $G^+$ in $G$.
		
		For the final stage of the proof, suppose that $\left\langle x_1G^+,x_2G^+,\dots,x_kG^+\right\rangle$ is a finitely generated subgroup of $G/G^+$. By setting $H=\left\langle x_1,x_2,\dots,x_k\right\rangle \le G$, we have $$\left\langle x_1G^+,x_2G^+,\dots,x_kG^+\right\rangle = HG^+/G^+\cong H/H\cap G^+.$$ 
		Since $H$ is finitely generated, from definitions we have $H^0_H=H^+$. Also, it follows from \cite[Lemma 5.2]{weh_1973} that $H^0_H$ is a normal subgroup of finite index of $H$. Hence,  $H/H^+$ is a finite group. Recall that by definition, we have  $G^+=\cup_Y(G\cap Y)^{0}_Y$, where $Y$ runs over the set of all finitely generated subgroups of $D^*$. Since $H$ is finitely generated subgroup of $G$, it follows that $G^+\ge (G\cap H)^0_H=H^0_H=H^+$. Consequently, $H^+\le G^+\cap H$ which implies that  $HG^+/G^+\cong H/G^+\cap H$ is finite. Hence, the subgroup  $\left\langle x_1G^+,x_2G^+,\dots,x_kG^+\right\rangle$ is finite, which shows that $G/G^+$ is locally finite. 
	\end{proof} 
	
	The following theorem can be considered as an analogous version of the Tits Alternative for weakly locally finite division rings.
	\begin{theorem} \label{theorem_2.8} 
		Let $D$ be a weakly locally finite division ring.  For a subgroup $G$ of $D^*$, the following statements are equivalent:
		\begin{enumerate}
			\item $G$ contains no non-cyclic free subgroups.
			\item $G$ is  locally solvable-by-finite.
			\item $G^+$ is abelian.
			\item $G$ is abelian-by-locally finite.
			\item $G$ is locally abelian-by-finite.
			\item Every finitely generated subgroup of $G$ satisfies a group identity. 
		\end{enumerate}
	\end{theorem}
	
	\begin{proof} Let $F$ be the center of $D$. 
		
		$1. \Rightarrow 2.$  Assume that $G$ contains no non-cyclic free subgroups. Let $H=\langle S\rangle $ be the subgroup of $G$ generated by a finite subset $S$ of $G$, and $K$  the division subring of $D$ generated by $S$. Then, $K$ is centrally finite. Setting $m=[K:Z(K)]$, where $Z(K)$ is the center of $K$, we may view $H$ as a finitely generated subgroup of $\mathrm{GL}_{m}(Z(K))$. By the Tits Alternative,  $H$ is solvable-by-finite, which yields that $G$ is locally solvable-by-finite.
		
		$2. \Rightarrow 3.$   Assume that $G$ is locally solvable-by-finite. For any finitely generated subgroup $Y$ of $D^*$,  set $X=G\cap Y$. We claim that  the connected component $X_Y^0$ of $X$ is an abelian group. Since $G$ is locally solvable-by-finite, so is $X$. Let $K$ be the division subring of $D$ generated by a finite set of generators  of $Y$. It follows from \cite[Lemma 10.12]{weh_1973} that $X$ is solvable-by-finite. In other words, there is a solvable normal subgroup $A$ of $X$ such that $X/A$ is finite. Now, for the subgroup $Y$, we use the corresponding symbols $D_Y, m_Y, F_Y$ as defined before. With respect to the Zariski topology defined on ${\mathrm M}_{m_Y}(D_Y)$, the subgroups $A\leq X\leq Y$ are $CZ$-groups. If we denote by $\bar{A}$ the closure of $A$ in $X$, then $\bar{A}$ is a solvable normal subgroup of $X$ (\cite[Lemma 5.1]{weh_1973}). According to \cite[Lemma 5.3]{weh_1973}, we conclude that $X_Y^0\subseteq \bar{A}$, from which it follows that $X_Y^0$ is also solvable. If $\bar{F}$ is the algebraic closure of $F_Y$, then ${\mathrm M}_{m_Y}(F_Y)\otimes_{F_Y}\bar{F}\cong {\mathrm M}_{m_Y}(\bar{F})$. Viewing $X_Y^0$ as a subgroup of $\mathrm{GL}_{m_Y}(\bar{F})$, we may apply the Lie-Kolchin Theorem (\cite[Lemma 5.8]{weh_1973}) to conclude that $X_Y^0$ is triangularizable, and so $(X_Y^0)'$ is unipotent. Since the only unipotent subgroup of $D^*$ is 1, we obtain that $(X_Y^0)'=1$, which says that $X_Y^0$ is abelian, as claimed. Finally, it is evident that $(G^+)'=\cup_Y(X_Y^0)'=1$ which implies that $G^+$ is  abelian, proving (3). 
		
		$3. \Rightarrow 4.$  This implication follows immediately by applying Theorem~ \ref{th:3.3}. 
		
		$4. \Rightarrow 5.$  Assume that $H$ is a finitely generated subgroup of $G$. Since $(4)$ holds, there exists an abelian normal subgroup $A$ of $G$ such that $G/A$ is locally finite. It follows that $H/H\cap A\cong HA/A$ is a finite group, which implies that $H$ is abelian-by-finite.
		
		$5. \Rightarrow 6.$ Assume that $(5)$ holds and $H$ is a finitely generated subgroup of $G$. Then, there is an abelian normal subgroup $A$ of $H$ such that $H/A$ is a finite group, say of order $m$. It follows that  $H$ satisfies the group identity $x^{-m}y^{-m}x^my^m=1$.
		
		Finally, the implication  $6. \Rightarrow 1.$  is obvious, and this completes the proof. 
	\end{proof}
	Theorem \ref{theorem_2.8} may have different applications, especially in the problem on the existence of  non-cyclic free subgroups in skew linear groups. As an example of such an application, we consider the following case. 
	
	For any subgroup $H$ of a group $G$, recall that the symbol $H^G$ denotes the normal closure of $H$ in $G$; that is, $H^G=\left\langle g^{-1}Hg:g\in G\right\rangle $. The next theorem describes the structure of almost subnormal subgroups of a subgroup in a weakly locally finite division ring. Recall that according to Hartley \cite{1989_Hartley}, a subgroup $H$ of a group $G$ is \textit{almost subnormal} in $G$ if there exists such a sequence of subgroups
	$$H=H_0\leq H_1\leq H_2\leq\dots \leq H_n=G$$
	that for any $0\leq i\leq n-1$, either $H_i$ is normal in $H_{i+1}$ or the index of $H_i$ in $H_{i+1}$ is finite. From definitions, we see that every subnormal subgroup of a group is almost subnormal. As it was shown in \cite{deo-bien-hai-20} and \cite{ngoc_bien_hai_17},  there exist infinitely many division rings whose multiplicative groups contain almost subnormal subgroups that are not subnormal. Moreover, in a weakly locally division ring, the problem on the existence of non-cyclic free subgroups in an almost subnormal subgroup of a group will be reduced to that of some its normal subgroup, as we see in the following theorem.
	
	\begin{theorem}\label{theorem_2.7}
		Let $D$ be a weakly locally finite division ring. Assume that $H\leq G$ are non-central subgroups of $D^*$; that is, they are not contained in the center of $D$. If $H$ is an almost subnormal subgroup of $G$, then $H$ contains a non-cyclic free subgroup if and only if so does $H^G$.
	\end{theorem}
	\begin{proof}
		The ``only if'' part is  obvious. For the ``if" part, in view of Theorem \ref{theorem_2.8}, it suffices to show that $H^G$ is locally solvable-by-finite. Assume that $A$ is a finitely generated subgroup of $H^G$; that is, $A=\left\langle g_1^{-1}h_1g_1,\dots,g_k^{-1}h_kg_k\right\rangle $ for some elements $g_1,\dots, g_k\in G$  and $h_1,\dots, h_k\in H$. Then, the division subring $E=\mathbb{P}(G_1)$ of $D$ generated by the subgroup $G_1=\left\langle g_1, h_1,\dots,g_k,h_k\right\rangle$ over the prime subfield $\mathbb{P}$ of $D$ is centrally finite, and we set $n=[E:Z(E)]$, where $Z(E)$ is the center of $E$. Since $H$ is an almost subnormal subgroup of $G$, it follows that  $H_1:=H\cap G_1$ is almost subnormal in $G_1$. If $H$ contains no non-cyclic free subgroups, then so does $H_1$. By Theorem \ref{theorem_2.8}, $H_1$  is locally solvable-by-finite. Viewing $H_1$ and $G_1$ as subgroups of $\mathrm{GL}_n(Z(E))$, we may apply \cite[point 3]{wehrfritz_93} to conclude that $H_1^{G_1}$ is solvable-by-finite. It is clear that $A\subseteq H_1^{G_1}$, hence $A$ is also solvable-by-finite, and this completes the proof. 
	\end{proof}
	Notice that in the above theorem, although $H$ is an almost subnormal subgroup of $G$, the subgroup $H^G$ is even normal. Therefore, Theorem \ref{theorem_2.7} has reduced the problem on the existence of non-cyclic free subgroups in an almost subnormal subgroup of a group to that of a normal subgroup. This observation would be useful in many circumstances. For a good demonstration, in the following corollary of Theorem~ \ref{theorem_2.7}, we give a short proof for one result previously  obtained in  \cite{ngoc_bien_hai_17}.
	\begin{corollary}[{\cite[Theorem 2.7]{ngoc_bien_hai_17}}]\label{corollary_2.6}
		Let $D$ be a weakly locally finite division ring, and $G$ an almost subnormal subgroup of $D^*$. If $G$ is non-central, then it contains a non-cyclic free subgroup.
	\end{corollary}
	
	\begin{proof}
		Deny of the conclusion, assume that $G$ contains no non-cyclic free subgroups. In virtue of Theorem \ref{theorem_2.7}, $N:=G^{D^*}$  does not contain non-cyclic free subgroups. By Theorem \ref{theorem_2.8}, $N$ contains an abelian normal subgroup $A$ such that $N/A$ is locally finite. Since $A$ is an abelian subnormal subgroup of $D^*$, by Stuth's theorem \cite[Theorem 4]{1964_stuth}, $A$ is contained in the center $F$ of $D$. Because $N\not\subseteq F$, by applying Cartan-Brauer-Hua's theorem \cite[Theorem 2]{1949_hua} (see also  \cite[(13.17)]{lam}), we have $F(N)=D$, from which it follows $Z(N)=N\cap F^*$, and consequently,  $A\le Z(N)$. Since $N/A$ is locally finite, it follows that  $N/Z(N)$ is locally finite, and in view of \cite[Lemma 3]{2003_akbari},  $N'$ is locally finite too. Being a periodic subnormal subgroup of $D^*$, by \cite[Theorem~ 8]{her_1978}, the subgroup $N'$ is contained in $F$. Therefore, $N$ is a solvable normal subgroup of $D^*$, and it is contained in $F$ by Stuth's theorem, a contradiction. 
	\end{proof}
	
	\section{Weakly locally finite division rings and Stewart's division rings}
	
	As we have proved in Section 2,  every Stewart's division ring is weakly locally finite. As a further application of Theorem \ref{theorem_2.8}, we now show that the question on the existence of non-cyclic free subgroups in weakly locally finite division rings can be reduced to that in Stewart's division rings. At this point of view, the following theorem is useful. 
	
	\begin{theorem}\label{theorem_2.9}
		Let $D$ be a weakly locally finite division ring with center $F$, and $G$ a subgroup of $D^*$. If $G$ contains no non-cyclic free subgroups, then $F(G)$ is a Stewart's division ring.
	\end{theorem}
	
	\begin{proof}
		Assume that a subgroup $G$ of $D^*$ contains no non-cyclic free subgroups. In view of Theorem \ref{theorem_2.8}, there exists an abelian normal subgroup $A$ of $G$ such that $G/A$ is locally finite. Then, the division subring $K=F(A)$ of $D$ is  a field, which is clearly normalized by $G$. Accordingly, if we set $\Delta=K[G]$, then any element $x\in\Delta$ can be written in the form $x=k_1g_1+\cdots+k_ng_n$ for some $k_1,\dots,k_n\in K$ and $g_1,\dots,g_n\in G$. If $H$ is the subgroup of $G$ generated by  $g_1, g_2, \cdots, g_n$, then the local finiteness of $G/A$ implies that $AH/A$ is a finite group. Let $\{y_1,\dots,y_m\}$ be a transversal of $A$ in $AH$, and let 
		$$R=Ky_1+\cdots+Ky_m.$$ 
		We note that $Ky_i=y_iK$ for all $i$, so $R$ is indeed a ring. Being contained in $D$, the ring $R$ is a domain, which is clearly finite dimensional over $K$. It follows that $R$ is a division ring (\cite[Exersice 13.2]{lam}), which is contained in $\Delta$, and so $x^{-1}\in R\subseteq \Delta$. Hence, $\Delta$ is a division ring, and consequently, $\Delta=F(G)$. Finally, because $G/A$ is locally finite, it follows that $\Delta$ is locally finite dimensional vector space over $K$. In other words, $\Delta$ is a $K$-Stewart's division ring. 
	\end{proof}	
	
	For the sake of a refinement, let us assume further that $G$ is algebraic over $F$. Under this condition, we might obtain a stronger result.
	
	\begin{corollary}\label{corollary_algebraic}
		Let $D$ be a weakly locally finite division ring with center $F$, and $G$ a subgroup of $D^*$. Assume that $G$ is algebraic over $F$. If $G$ contains no non-cyclic free subgroups, then $F(G)$ is a locally finite division ring.
	\end{corollary}
	
	\begin{proof}
		Let us keep all symbols and notations as in the proof of Theorem \ref{theorem_2.9}. Our hypothesis asserts that the subfield $K=F(A)$ is an algebraic extension of $F$. We claim that $F(G)$ is algebraic over $F$. Indeed,  take a nonzero element $x\in F(G)$. The previous theorem says that $F(G)$ is a $K$-Stewart's division ring. In view of Lemma~ \ref{lem:2.2},   $[K(x):K]_L<\infty$, from which it follows that $x$ is left algebraic over $K$. Let $p(t)=t^n+a_{n-1}t^{n-1}+\cdots+a_1t+a_0\in K[t]$ is the left minimal polynomial of $x$. Then, we have
		$$x^n+a_{n-1}x^{n-1}+\cdots+a_1x+a_0=0.\eqno(1)$$
		Since $x$ normalizes $K$, for each $1\leq i\leq n$, we can find an element $b_i\in K$ such that $a_ix=xb_i$. The condition that $G$ is algebraic over $F$ implies in particular that the field $K$ is algebraic over $F$. Hence, $E=F(a_1, b_1,\dots,a_n,b_n)$  is a finite field extension over $F$. Clearly, $E$ is normalized by $x$. This implies that the elements $1, x, x^2, \dots, x^i, \dots$ generate $E[x]$ as a left $E$-vector space. From (1), we have
		$$x^n=-a_{n-1}x^{n-1}-\dots-a_1x-a_0,$$
		and it follows that for each $i\ge n, x^i$ can be expressed as a left linear combination over $E$ of elements $1, x, x^2, \dots, x^{n-1}$. As a result, we conclude that $\{1, x, x^2, \dots, x^{n-1}\}$ is a basis of $E[x]$ over $E$. The finiteness of $[E:F]$ and $[E[x]:E]_L$ implies that $[E[x]:F]<\infty$, from which it follows that $x$ is algebraic over $F$. Hence, $F(G)$ is algebraic over $F$, as claimed. Finally, since $F(G)$ is both weakly locally finite and algebraic over $F$, we may apply \cite[Theorem~5]{deo19} to conclude that it is a locally finite division ring. 
	\end{proof}
	
	Before closing this section, we record some additional results regarding locally solvable subgroups in a weakly locally division ring.
	
	\begin{proposition}\label{proposition_4.1}
		Every locally nilpotent subgroup in a weakly locally finite division ring is center-by-locally finite.
	\end{proposition}
	
	\begin{proof}
		Assume that $D$ is a weakly locally fintie division ring, and $G$ is an arbitrary locally nilpotent subgroup of $D^*$. For any finitely generated subgroup $Y$ of $G$, the subgroup $X=G\cap Y$ is $CZ$-group with the Zariski topology on ${\mathrm M}_{m_Y}(F_Y)$. Let $R$ be the subring of $D$ generated by all (finitely many) generators and their inverses of $Y$ over the prime subfield of $D$. Then, we may view $X$ as a subgroup of $R^*$. It follows that $X$ is nilpotent (\cite[Theorem 4.23]{weh_1973}), and $X/Z(X)$ is finite (\cite[Theorem 3.13]{weh_1973}). For any $x\in X$, we have $|x^X|=[X:C_X(x)]<\infty$, which implies that $X$ is an $FC$-group, so $X^0\subseteq Z(X)$ (\cite[Lemma 5.5]{weh_1973}). Now, we have 
		$$G^+=\cup_Y(G\cap Y)^0=\cup_YX^0\subseteq\cup_YZ(X)=\cup_Y(Z(G)\cap X) \subseteq Z(G).$$
		In view of Theorem \ref{th:3.3}, we conclude that $G$ is center-by-locally finite. 
	\end{proof}
	
	In view of \cite[1.4.13]{shi-weh_1986}, we see that a locally solvable subgroup of a locally finite division ring need not be solvable. However, any locally nilpotent or locally supersolvable subgroups of a weakly locally finite division ring are always solvable.
	
	\begin{corollary}\label{corollary_4.2}
		Every locally nilpotent subgroup in a weakly locally finite division ring is metabelian. 
	\end{corollary}
	
	\begin{proof}
		Keeping the same notations $D$ and $G$ as in Proposition \ref{proposition_4.1}, we have $G/Z(G)$ is locally finite. Finally, by applying \cite[3.2.2 and 3.2.6]{shi-weh_1986}, we conclude that  $G'$ is abelian, which means that $G$ is metabelian. 
	\end{proof}
	
	If the conditions ``locally solvable'' and ``solvable'' are equivalent for linear groups, then it is not the case for skew linear groups. For instance, Wehrfritz has constructed a locally finite division ring in which there exist locally solvable subgroups that are not solvable (\cite[1.4.13]{shi-weh_1986}). Whereas, locally nilpotent or locally supersolvable subgroups in a weakly locally finite division ring must be solvable, as we see in the next result.
	
	\begin{proposition}\label{proposition_4.3}
		Every locally supersolvable subgroup in a weakly locally finite division ring is solvable of class at most $3$.
	\end{proposition}
	
	\begin{proof} 
		Keeping the same notations as above, by hypothesis, we conclude that $G$ contains no non-cyclic free subgroups. Accordingly, we can view $G$ as a subgroup of a Stewart's division ring according to Theorem \ref{theorem_2.9}. By \cite[Theorem 2.3]{stewart_87}, there exists a locally nilpotent normal subgroup $A$ of $G$ such that $G/A$ is a periodic abelian group. Due to Corollary \ref{corollary_4.2},  $A$ is metabelian, which implies that $G$ is solvable with subnormal series $1\leq A' \leq A \leq G$. 
	\end{proof}
	\section{Maximal subgroups of almost subnormal subgroups}
	
	The existence of non-cyclic free subgroups in a maximal subgroup of  $D^*$ was first studied by M. Mahdavi-Hezavehi in \cite{2001_mah}.
	Let $D$ be a division ring with center $F$, and $G$ a non-central subnormal subgroup of $D^*$. Assume that $M$ is a non-abelian maximal subgroup of $G$. Recently, Hai and Tu \cite{2016_hai-tu} study the problem on the existence of non-cyclic free subgroups in $M$ for the case when $D$ is locally finite. The result obtained in \cite[Theorem 3.4]{2016_hai-tu} generalizes the result previously proved by Mahdavi-Hezavehi in \cite{2001_mah} for the case when $D$ is centrally finite and $G=D^*$. In \cite{2019_hai-khanh}, we study substantially the same problem for the case in which $G$ is assumed to be an almost subnormal subgroup in  a locally finite division ring $D$. In this section, we extend this result for the case when $D$ is weakly locally finite. The essential fact we have successfully proved is that we can reduce the considered  weakly locally finite division ring $D$ to the case of locally finite division ring $D$ provided $M$ is algebraic over the center $F$ of $D$.  
	
	In order to use some Wehrfritz's results, we need the following concept, which is extremely useful in the sequel. Let $R$ be a ring, $S$ a subring of $R$, and $G$ a subgroup of $R^*$ normalizing $S$ such that $R=S[G]$. Assume further that $N=G\cap S$ is a normal subgroup of $G$ and that $R=\bigoplus_{t\in T} tS$, where $T$  is some (and hence any) transversal of $N$ in $G$. Then, the resulting system $(R, S, G, G/N)$ is called a \textit{crossed product}\index{crossed product}. Also, we sometimes say that $R$ is a crossed product of $S$ by $G/N$ or, more briefly, of $S$ over $N$  (see \cite{wehrfritz_91} or \cite[p.23]{shi-weh_1986}). Firstly, we record a few auxiliary lemmas.
	
	\begin{lemma}\label{lemma_crossed}
		Let $R$ be a ring, and $G$ a subgroup of $R^*$. Assume that $F$ is a central subfield of $R$ and that $A$ is a maximal abelian subgroup of $G$ such that $K=F[A]$ is normalized by $G$. Then $F[G]$ is a crossed product of $K$ by $G/A$. In addition, if $R=F[G]$ and $K$ is a field, then it is a maximal subfield of $R$.
	\end{lemma}
	
	\begin{proof} 	
		Since $K$ is normalized by $G$, it follows that $F[G]=\sum\nolimits_{g \in T}{Kg}$ for every transversal $T$ of $A$ in $G$. Accordingly, to establish that $(R, S, G, G/A)$ forms a crossed product, it suffices to prove that every finite subset $\{g_1,g_2,\dots,g_n\}$ of $T$ is  linearly independent over $K$. Assume by contrary that there exists  a non-trivial relation 
		$$k_1g_1+k_2g_2+\dots+k_ng_n=0, \text{ where } k_1, k_2, \dots k_n\in K.$$
		Clearly, we can suppose that all the $k_i$'s are non-zero and that $n$ is minimal. The case $n=1$ is obviously trivial, so  suppose that $n>1$. As the cosets $Ag_1$ and $Ag_2$ are disjoint, we know that $g_1^{-1}g_2\not\in A=C_G(A)$. Accordingly, there exists an element $x\in A$ for which $g_1^{-1}g_2x\ne xg_1^{-1}g_2$. For each $1\leq i\leq n$, if we set $x_i=g_ixg_i^{-1}$, then $x_1\ne x_2$. Since $G$ normalizes $K$, it follows that $x_i\in K$ for all $1\leq i\leq n$. Now, we have 
		$$(k_1g_1+\dots+k_ng_n)x-x_1(k_1g_1+\dots+k_ng_n)= 0.$$
		By definition of the $x_i$'s, we deduce that $x_ig_i=g_ix$, and so $x$, $x_i\in K$ for all $i$. By the fact that $K=F[A]$ is commutative, last equality reveals
		$$\left( {{x_2} - {x_1}} \right){k_2}{g_2} +  \dots  + \left( {{x_n} - {x_1}} \right){k_n}{g_n} = 0,$$
		which is a non-trivial relation (since $x_1\ne x_2$) with less than $n$ summands, contrasting with the minimality of $n$. As a result, we obtain the desired fact that $T$ is linearly independent over $K$.
		
		Regarding the last assertion of our lemma,  we assume that $R=F[G]$ and that $K$ is a field. If we set $L=C_R(K)$, then every element $y\in L$ may be written in the form
		$$y=l_1m_1+l_2m_2+\dots+l_tm_t,$$ 
		where $l_1,l_2,\dots,l_t\in K$ and $m_1,m_2,\dots,m_t\in T$. Take an arbitrary element  $a\in A$, by the normality of $A$ in $M$, there exist $a_i\in A$ such that $m_ia=a_im_i$ for all $1\leq i\leq t$. Since $ya=ay$, it follows that
		$$ (l_1a_1-l_1a)m_1+(l_2a_2-l_2a)m_2+\dots+(l_ta_t-l_ta)m_t=0.$$ 
		As $\{m_1,m_2,\dots,m_t\}$ is linearly independent over $K$, the outcome is that $$a=a_1=\dots=a_t.$$
		Consequently, $m_ia=am_i$ for all $a\in A$; thus, $m_i\in C_M(A)=A$ for all $1\leq i\leq t$. The consequence of this fact is that $y\in K$, yielding $L=K$. The last fact ensures the maximality of the field $K$, and the proof is now complete. 
	\end{proof}
	
	\begin{lemma}[{\cite[3.2]{wehrfritz_91}}]\label{lemma_4.2}
		Let $R$ be a ring, $J$ a subring of $R$, and $H\leq K$ subgroups of the unit group $R^*$  normalizing $J$ such that $R$ is the ring of right quotients of $J[H]\leq R$ and $J[K]$ is a crossed product of $J[B]$ by $K/B$ for some normal subgroup $B$ of $K$. Then, $K=HB$.
	\end{lemma}
	
	The next theorem, which serves as the main result of the present section, may be considered as a necessary condition for a weakly locally finite division ring to be a Stewart's division ring. 
	
	\begin{theorem}\label{theorem_main2}
		Let $D$  be a weakly locally finite division ring with center $F$, and $G$  a non-central almost subnormal subgroup of $D^*$. Assume that $M$ is a non-abelian maximal subgroup of $G$. If $M$ contains no non-cyclic free subgroups, then $D$ is a Stewart's division ring. Moreover,  there exists a maximal subfield $K$ of $D$  containing $F$ such that $K/F$ is a Galois extension, $N_G(K^*)=M $, $K^*\cap G\unlhd M$, and $M/K^*\cap G\cong\mathrm{Gal}(K/F)$ is a locally finite simple group.
	\end{theorem}
	
	\begin{proof} 
		Assume that $M$ contains no non-cyclic free subgroups. With reference to Theorem \ref{theorem_2.8}, we can conclude that $M$ contains a maximal subgroup $A$ with respect to the property: $A$ is an abelian normal subgroup of $M$ such that $M/A$ is locally finite. Then, the subfield $K=F(A)$ of $D$ is normalized by $M$. A similar argument used in the proof of Theorem \ref{theorem_2.9} reveals that $E:=K[M]$ is a division ring, which is locally finite over $K$. Since $M$ is maximal in $G$, either $E^*\cap G=M$ or $G\le E^*$. In the first case, one would have that $M$ is a non-central almost subnormal subgroup of $E^*$ containing no non-cyclic free subgroups, contradicting  Corollary \ref{corollary_2.6}. This contradiction forces the second case to occur and so $E=D$ by \cite[Corollary 2.3]{2019_khanh-hai}. It follows that $D$ is a $K$-Stewart's division ring.
		
		Let $L=C_D(K)$. Observe that $L$ is a division subring of $D$ normalized by $M$; hence $M$ is contained in the normalizer $N_{D^*}(L^*)$ of $L^*$ in $D^*$. Let us denote  $N_G(L^*)=G\cap N_{D^*}(L^*)$. The maximality of $M$ implies that $M\le N_G(L^*)\le G$,  so either  $G=N_G(L^*)$ or $M= N_G(L^*)$. 
		
		\bigskip
		
		\textit{Case 1. $G=N_G(L^*)$:}
		
		\bigskip
		
		Since $G$ is almost subnormal in $D^*$, it follows that $L^*\cap G$ is almost subnormal in $D^*$ too. The condition $G=N_G(L^*)$  yields that $G\le N_{D^*}(L^*)$, so $L^*\cap G$ normalizes $L$. In view of \cite[Corollary 2.3]{2019_khanh-hai}, either $L^*\cap G\subseteq F$ or $L=D$. It is clear that in both these possibilities,  we always have $A\subseteq F$ which implies that $K=F$. It follows that $D$ is a locally finite division ring,  and the conclusion of the theorem follows from \cite[Theorem~ 3.1]{2019_hai-khanh}. 
		\bigskip
		
		\textit{Case 2. $M= N_G(L^*)$:}
		
		\bigskip
		
		The condition  $M= N_G(L^*)$ implies that $L^*\cap G \le M$. It follows that $L^*\cap G$ is an almost subnormal subgroup of $L^*$ containing no non-cyclic free subgroups. It is obvious that $L$ is also weakly locally finite. In view of Corollary~ \ref{corollary_2.6}, we have $L^*\cap G\subseteq Z(L)$ which implies that  $L^*\cap G$ is an abelian normal subgroup of $M$.  Since $M/A$ is locally finite, it follows that $M/L^*\cap G$ is locally finite too. By the maximality of $A$ in $M$, one has $A=L^*\cap G$. Because we are in the case $L^*\cap G \le M$, it follows that $L^*\cap G=L^*\cap M$. Consequently,  $$A=L^*\cap M=C_D(A)\cap M =C_M(A),$$ which shows that $A$ is a maximal abelian subgroup of $M$. In virtue of Lemma \ref{lemma_crossed}, we conclude that $K$ is  a maximal subfield of $D$  and that $D$ is a crossed product of $K$ by $M/A$.
		
		Our next step is to show that $K/F$ is a Galois extension. Recall that $A$ is normal in $M$, so for any $a\in M$,  the mapping $\theta_a:K\to K$ given by $\theta_a(x)=axa^{-1}$ is well defined. It is a fairly simple matter to see that $\theta_a$ is an $F$-automorphism of $K$. Accordingly, the mapping $$\psi:M\to \mathrm{Gal}(K/F)$$ defined by $\psi(a)=\theta_a$ is a group homomorphism with $$\mathrm{ker}\psi=C_M(K^*)=C_D(K^*)\cap M=K^*\cap G.$$ Since $K[M]=D$ and $A$ is contained in $M$, it follows that $C_D(M)=F$. Therefore, the fixed field of $\psi(M)$ is $F$. From this, we conclude that $\psi$ is a surjective homomorphism, and $K/F$ is a Galois extension, and so $M/K^*\cap G\cong\mathrm{Gal}(K/F)$, which is a locally finite group.
		
		Now, we prove that $M/A$ is simple. For this purpose, suppose that $N$ is an arbitrary normal subgroup of $M$ properly containing $A$. Note that by the maximality of $A$ in $M$, we may assume further that $N$ is non-abelian. We assert that $R:=K[N]=D$. As $N$ is normal in $M$, we have $M\le N_G(R^*)\le G$ and so either $N_G(R^*)=M$ or $N_G(R^*)=G$. If the first case occurs, then $R^*\cap G\le M$ and, consequently, $R^*\cap G$ is a normal subgroup of $R^*$ containing no non-cyclic free subgroups. By the same argument used in the proof of Theorem \ref{theorem_2.9}, we conclude that $R$ is a weakly locally finite division ring. It follows from Corollary \ref{corollary_2.6} that $R^*\cap G$ is abelian, which contradicts  the choice of $N$. Thus, the case  $N_G(R^*)=G$ must occur, from which it follows that $R=D$, as asserted. In view of Lemma \ref{lemma_4.2}, we deduce that $M=AN=N$ because $A\le N$. This implies that $M/A$ is a simple group. Finally, the fact $A=K^*\cap G$ shows that $M/K^*\cap G$ is a locally finite simple group. 
	\end{proof}
	
	In the above theorem, if we assume further that $M$ is algebraic over $F$, then we might obtain the following stronger result. 
	
	\begin{corollary}
		Let $D$  be a weakly locally finite division ring with center $F$, and $G$ an almost subnormal subgroup of $D^*$. Assume that $M$ is a non-abelian maximal subgroup of $G$ and that $M$ is algebraic over $F$. If $M$ contains no non-cyclic free subgroups, then $[D:F]<\infty$, $F[M]=D$, and there exists a maximal subfield $K$ of $D$  containing $F$ such that $K/F$ is a Galois extension, $N_G(K^*)=M $, $K^*\cap G\unlhd M$, $M/K^*\cap G\cong\mathrm{Gal}(K/F)$ is a finite simple group, and $K^*\cap G$ is the Fitting subgroup of $M$.
	\end{corollary}
	\begin{proof} It follows by Theorem \ref{theorem_main2} that $D$ is a Stewart's division ring and $F(M)=D$. Therefore, we may apply Corollary \ref{corollary_algebraic} to conclude that $D$ is a locally finite division ring. Hence, the conclusions  follow from \cite[Theorem 3.1]{2019_hai-khanh}. 
	\end{proof}

	\textbf{Acknowledgements.} The authors would like to express the sincere gratitude to the referee.

\end{document}